\documentclass{amsart}

\usepackage{graphicx}

\newtheorem{theorem}{Theorem}[section]
\newtheorem{lemma}[theorem]{Lemma}
\newtheorem{proposition}[theorem]{Proposition}

\theoremstyle{definition}
\newtheorem{definition}[theorem]{Definition}
\newtheorem{example}[theorem]{Example}

\theoremstyle{remark}
\newtheorem{remark}[theorem]{Remark}

\numberwithin{equation}{section}

\usepackage{ulem}
\usepackage{amsrefs}

\usepackage{color}

\let\oldmarginpar\marginpar
\renewcommand\marginpar[1]{\-\oldmarginpar[\raggedleft\footnotesize #1]
{\raggedright\footnotesize #1}}

\usepackage[T1]{fontenc}
\usepackage{lmodern}
\usepackage[latin1]{inputenc}

\usepackage{setspace}
\usepackage{indentfirst}

\usepackage{multicol}
\usepackage{geometry}
\usepackage{hyperref}
\usepackage{amsmath,amssymb}
\usepackage{mathrsfs}
\usepackage{amsthm}

\theoremstyle{plain}

\newtheorem*{theorem*}{Theorem}
\newtheorem*{corollary*}{Corollary}

\newcommand{\C}{\mathbb{C}}

\newcommand{\R}{\mathbb{R}}
\newcommand{\Z}{\mathbb{Z}}

\newcommand{\A}{\mathcal{A}}
\newcommand{\T}{\mathcal{T}}

\renewcommand{\emptyset}{\varnothing}
\newcommand{\Lie}{\text{Lie}}

\DeclareMathOperator{\Hom}{Hom}

\begin{document}

\title{BV operators of the Gerstenhaber algebras of holomorphic polyvector fields on toric varieties}

\author{Yang Deng}
\address{School of Mathematics and Statistics, Wuhan University, Wuhan, 430072, China}
\email{yangdeng@whu.edu.cn}

\author{Wei Hong}
\address{School of Mathematics and Statistics, Wuhan University, Wuhan, 430072, China}
\address{Hubei Key Laboratory of Computational Science, Wuhan University, Wuhan, 430072, China}
\email{hong\textunderscore  w@whu.edu.cn}

\keywords{toric variety, holomorphic polyvector fields, Gerstenhaber algebra, BV operator}

\begin{abstract}
The vector space of holomorphic polyvector fields on any complex manifold has a natural Gerstenhaber algebra structure. In this paper, we study BV operators of the Gerstenhaber algebras of holomorphic polyvector fields on smooth compact toric varieties. We give a necessary and sufficient condition  for the existence of BV operators of the Gerstenhaber algebra of holomorphic polyvector fields on any smooth compact toric variety.
\end{abstract}

\maketitle


\section{introduction}\label{sect-Intr}

Polyvector fields appeared in the recent work of many mathematicians. 
Barannikov and Kontsevich \cite{B-K 98} show that polyvector fields play important roles in deformation theory and mirror symmetry. 
Polyvector fields also appeared in the work of Bandiera, Sti\'{e}non and Xu \cite{B-S-Xu 19}.
In differential geomtry, holomorphic polyvector fields appeared in the work of Hitchin  \cite{Hitchin 11}. 
It is well known that the vector space of polyvector fields on a smooth manifold $M$ has 
a Gerstenhaber algebra structure by the wedge product and the Schouten bracket of polyvector fields.  In the same way, the vector space of holomorphic polyvector fields on a complex manifold $X$ becomes a Gerstenhaber algebra.
And the Gerstenhaber algebra structure of polyvector fields on the algebraic torus is studied in  \cite{M-R 19}. 

This paper is a subsequent work of \cite{Hong 20}. For any smooth compact toric variety $X$,  the vector space of holomorphic polyvector fields on $X$ has been given an explicit description in \cite{Hong 20}. 
In this paper, we study Batalin-Vilkovisky operators for the Gerstenhaber algebras of holomorphic polyvector fields on smooth compact toric varieties. 
We give a necessary and sufficient condition for the existence of BV operators. 
 And we also give the explicit forms of the BV operators.

{\bf Acknowledgements.} 
Deng would like to thank Prof. Yuping Tu and Prof. Fengwen An for their help during his study in Wuhan University.
Hong's research is partially supported by NSFC grant 12071241 and NSFC grant 12071358.

\section{Preliminary}
\subsection{Gerstenhaber algebra and Batalin-Vilkovisky algebra}
Let us recall some classical knowledge of Gerstenhaber algebra and Batalin-Vilkovisky algebra (BV algebra). One may consult \cite{Roger 09} and \cite{Xu 99}.

Let $\A=\bigoplus_{k\in\Z} \A^k$ be a $\mathbb{Z}$-graded vector space over a field 
$\mathbf{k}$. For any element $a\in \A^k$, we denote by $|a|=k$ the degree of $a$.
Let $\A[1]=\bigoplus_{k\in\Z} \A[1]^k$, where $\A[1]^k=\A^{k+1}$.
\begin{definition}
A graded vector spcae $\A=\bigoplus_{k\in\Z} \A^k$ is a Gerstenhaber algebra if one has:
\begin{enumerate}
\item 
An associative, graded commutative multiplication 
$$\A\times \A\rightarrow \A,$$
\i.e., for every $a,b,c\in \A$, we have
\begin{itemize}
\item $a\cdot(b\cdot c)=(a\cdot b)\cdot c$,
\item $a\cdot b=(-1)^{|a||b|}b\cdot a$.
\end{itemize}
\item
A graded Lie algebra bracket $[,]: \A[1]\times \A[1]\rightarrow \A[1]$, \i.e.,
we have $[\A^k, \A^l]\subseteq\A^{k+l-1}$, and for every $a,b,c\in \A$,
\begin{itemize}
\item $[a, b]=-(-1)^{(|a|-1)(|b|-1)}[b,a]$ (graded anti-symmetry),
\item $(-1)^{(|a|-1)(|c|-1)}[[a,b],c]+(-1)^{(|b|-1)(|a|-1)}[[b,c],a]+(-1)^{(|c|-1)(|b|-1)}[[c,a],b]=0$ (graded Jacobi identity).
\end{itemize}
\item
The multiplication and bracket satisfying the Leibniz relation, \i.e.,
for every $a,b,c\in \A$,
$$[a, b\cdot c]=[a,b]\cdot c+(-1)^{(|a|-1)|b|}b\cdot[a,c].$$
\end{enumerate}
\end{definition}

We have the following examples of Gerstenhaber algebras.
\begin{example}
Let $M$ be a smooth real manifold. Let 
$$\A^k=\begin{cases}
\Gamma(M,\wedge^k TM),\quad k>0;\\
C^{\infty}(M), \quad k=0;\\
0,\quad k<0.
\end{cases}$$
Then $\A=\bigoplus_{k\in\Z} \A^k$ is a $\mathbb{Z}$-graded vector space over $\R$.
Let the multiplication on $\A$ be the wedge product on $\Gamma(M,\wedge^k TM)$.
The bracket $[,]$ is given by the Schouten bracket of polyvector fields.
Then $(\A, \wedge, [,])$ is a Gerstenhaber algebra.
\end{example}

\begin{example}\label{poly-Gerst}
Let $X$ be a complex manifold. 
Let $\A^k=H^0(X,\wedge^k \T_X)$ be the vector space of holomorphic $k$-vector fields on $X$ for $k>0$. Let $\A^0=H^0(X,\mathcal{O}_X)$ be the vector space of holomorphic functions on $X$. And let $\A^k=0$ for $k<0$. 
Then $\A=\bigoplus_{k\in\Z} \A^k$ is a $\mathbb{Z}$-graded vector space over $\C$.
Let the multiplication on $\A$ be the wedge product of polyvector fields.
The bracket $[,]$ is given by the Schouten bracket of polyvector fields.
Then $(\A, \wedge, [,])$ is a Gerstenhaber algebra.
\end{example}

Next we give two equivalent definitions of Batalin-Vilkovisky algebras.
\begin{definition}
A graded vector space $\A=\bigoplus_{k\in\Z} \A^k$ over $\mathbf{k}$ is a 
Batalin-Vilkovisky algebra (BV algebra) if 
\begin{enumerate}
\item $\A$ is an associative, graded commutative algebra, i.e., there is an associative, graded commutative multiplication on $\A$.
\item
There is a differential $\mathbf{k}$-linear operator $D: \A\rightarrow \A$ of order $2$ and degree $-1$. 
Here $D$ is a differential operator of order $2$ in the sense that, for all $a,b,c\in \A$, 
\begin{gather}
D(abc)=D(ab)c+(-1)^{|a|}a D(bc)+(-1)^{|b|(|a|+1)}bD(ac)
-D(a)bc-(-1)^{|a|}aD(b)c-(-1)^{|a|+|b|}ab D(c)
\end{gather}
\item $D^2=0$.
\end{enumerate}
\end{definition}

A BV algebra $(\A, \cdot, D)$ is always a Gerstenhaber algebra with the same multiplication, and the bracket is given by
\begin{equation*}
[a,b]=(-1)^{|a|}(D(ab)-D(a)b-(-1)^{|a|}aD(b))
\end{equation*}
for all $a,b\in \A$.

The following is an equivalent definition of BV algebra.
\begin{definition}
Given a Gerstenhaber algebra $(\A, \cdot, [,])$, if there is exist a degree $-1$ 
$\mathbf{k}$-linear operator $D:\A\rightarrow \A$ such that
\begin{enumerate}
\item
\begin{equation}
[a,b]=(-1)^{|a|}(D(ab)-D(a)b-(-1)^{|a|}aD(b))
\end{equation}
for any $a,b\in \A$,
\item $D^2=0$,
\end{enumerate}
then the Gerstenhaber algebra $(\A,\cdot, [,])$  is said to be exact, 
and $D$ is called a generating operator. 
\end{definition}
That a Gerstenhaber algebra $(\A,\cdot, [,])$  is exact is equivalent to that $(\A,\cdot, D)$ is a BV algebra. And the generating operator $D$ is also called a BV operator.

\subsection{Holomorphic polyvector fields on toric varieties}
Recall that a toric variety \cite{Cox} is an irreducible variety $X$ such that
\begin{enumerate}
\item $T=(\C^*)^n$ is a Zariski open set of $X$, 
\item the action of $T=(\C^*)^n$ on itself extends to an action of $T=(\C^*)^n$ on $X$.
\end{enumerate}
One may consult \cite{Oda} and \cite{Fulton} for more details of toric varieties. 

Let $N=\Hom(\C^*,T)\cong\Z^n$ and $M=\Hom(T, \C^*)$.
Then $M\cong \Hom_\Z(N,\Z)$ and $N\cong \Hom_\Z(M,\Z)$.
Each element $m$ in $M$ gives rise to a character $\chi^m\in Hom(T,\C^*)$,
which can be also considered as a rational function on $X$. 
Let $N_{\R}=N\otimes_{\Z}\R\cong\R^n$, and $M_{\R}=M\otimes_{\Z}\R\cong\R^n$ 
be the dual space of $N_{\R}$.

Let $X=X_\Delta$ be a smooth compact toric variety associated with a fan $\Delta$ in $N_\R\cong\R^n$. 
The $T$-action on $X$ induces a $T$-action on $H^0(X, \wedge^k \T_{X})$, 
the vector space of holomorphic $k$-vector fields on $X$. 
Denote by $V_I^k$ the weight space corresponding to the character $I\in M$. 
We have
\begin{equation}\label{VIk-eqn}
H^0(X, \wedge^k\T_{X})=\bigoplus_{I\in M}V_I^k.
\end{equation}
Since $H^0(X, \wedge^k\T_{X})$ is a finite dimensional vector space, 
there are only finite elements $I\in M$ such that $V_I^k\neq 0$.

Some technical notations are necessary for us in this paper.

\begin{enumerate}
\item
Let $N_\C=N\otimes_{\Z}\C$. 
Since $T\cong N\otimes_\Z\C^*$, we have $\Lie(T)\cong N_\C$, 
where $\Lie(T)$ denotes the Lie algebra of $T$. 
We define a map $\rho:N_\C\rightarrow \mathfrak{X}(X)$ by
\begin{equation} \label{rho-def-eqn}
\rho:N_\C=N\otimes_\Z\C\cong \Lie(T)\rightarrow \mathfrak{X}(X),
\end{equation}
where $\Lie(T)\rightarrow \mathfrak{X}(X)$ is defined by the infinitesimal action of the Lie algebra $\Lie(T)$ on $X$. 
And the map $\rho:N_\C\rightarrow \mathfrak{X}(X)$ is determined by
$$\rho(x)(\chi^m)=\langle x,m\rangle\chi^m$$
for any $x\in N_\C$ and $m\in M$.
Since the action map $T\times X\rightarrow X$ is holomorphic, 
the images of $\rho$ are holomorphic vector fields on $X$.
By the abuse of notation, the induced maps
\begin{equation} \label{rho-def-eqn2}
\wedge^k N_\C\rightarrow H^0(X,\wedge^k\T_X): 
\quad x_1\wedge\ldots\wedge x_k\rightarrow \rho(x_1)\wedge\ldots\wedge\rho(x_k),
\quad\forall x_1,\ldots, x_k\in N_{\C},
\end{equation}
 are also denoted by $\rho$ for $2\leq k\leq n$.
 Let $W=\rho(N_\C)$, $W^k=\wedge^k W$ and $W^0=\C$. 
Then $W^k$ is a subspace of $H^0(X, \wedge^k\T_{X})$. 
And the map $\wedge^k N_\C\xrightarrow{\rho} W^k$ is an isomorphism.

\item
Let $\Delta(k)$ $(0\leq k\leq n)$ be the set of all $k$-dimensional cones in 
$\Delta$.
Suppose that $\alpha_t ~(1\leq t\leq r)$ are all one dimensional cones in $\Delta(1)$. 
Let  $e(\alpha_t)\in N$ $(1\leq t\leq r)$ be the corresponding primitive elements, i.e., the unique generator of $\alpha_t \cap N$. 
Let $P_\Delta$ be a polytope in $M_\R$ defined by
\begin{equation*}
P_\Delta=\bigcap_{\alpha_t\in \Delta(1)}\{I\in M_\R\mid \langle I,e(\alpha_t)\rangle\geq -1\}.
\end{equation*}
Since $X_{\Delta}$ is compact, $P_{\Delta}$ is a compact polytope in $M_{\R}$. 
Let 
\begin{equation*}
S_\Delta=\{I\mid I\in M\cap P_{\Delta}\}.
\end{equation*}
Then $S_{\Delta}$ is a non-empty finite set and $0\in S_{\Delta}$.
We denote by $P_\Delta(i)$ the set of all $i$-dimensional faces of the polytope $P_\Delta$.  Let
\begin{equation*}
S(\Delta,i)=\bigcup_{F_j\in P_\Delta(n-i)}\{I\in int(F_j)\cap M\}
\end{equation*}
 for $0\leq i\leq n$, where $int(F_j)$ denotes the relative interior of the face $F_j$.
Let
\begin{equation*}
 S_k(\Delta)=\bigcup_{0\leq i\leq k}S(\Delta,i)
 \end{equation*}
  for all $0\leq k\leq n$. 
  
\item  
For any $I\in S_{\Delta}$, there exists a unique face $F_I(\Delta)$ of $P_{\Delta}$ 
such that $I\in int(F_I(\Delta))\cap M$. 
Let
\begin{equation*}
E_I(\Delta)=\{e(\alpha_t)\mid\alpha_t \in\Delta(1)~\text{and}~\langle I,e(\alpha_t)\rangle=-1\}.
\end{equation*} 
Suppose that $E_I(\Delta)=\{e(\alpha_{s_1}),\ldots e(\alpha_{s_l})\}$, 
where $1\leq s_1<\ldots<s_l\leq r$.
Let $F^{\perp}_I(\Delta)\subseteq N_{\R}$ be the normal space of $F_I(\Delta)$,
defined by
\begin{equation*}
F^{\perp}_I(\Delta)=\sum_{1\leq t\leq l} \R\cdot e(\alpha_{s_t}).
\end{equation*}
And we set
\begin{equation*} 
F^{\perp}_I(\Delta)=0 \quad\text{if}\quad E_I(\Delta)=\emptyset.
\end{equation*}

\item
Suppose that $I\in S(\Delta,i)$. 
Then we have $\dim F_I(\Delta)=n-i$ and $\dim F^{\perp}_I(\Delta)=i$.
Considering $F_I^{\perp}(\Delta)\otimes\C$ as a subspace of $N_{\C}$,
we have $\wedge^i (F^{\perp}_I(\Delta)\otimes\C)\cong\C$.

Let $\hat{F}_I(\Delta)$ be the affine space in $M_{\C}$ with the lowest dimension such that $F_I(\Delta)\subseteq \hat{F}_I(\Delta)$, where $F_I(\Delta)\subset M_{\R}$ is considered as a subset of $M_{\C}$ in the natural way. Then we have
\begin{equation}
\hat{F}_I(\Delta)=\{I+\alpha\mid \alpha\in (F^{\perp}_I(\Delta)\otimes\C)^{\perp} \},
\end{equation}
where $(F^{\perp}_I(\Delta)\otimes\C)^{\perp} \subseteq M_{\C}$ is the annihilator of 
 $F_I^{\perp}(\Delta)\otimes\C\subseteq N_{\C}$.
 Especially, if $I\in M$ is a vertex of  $P_{\Delta}$, then we have $\hat{F}_I(\Delta)=\{I\}$.

\item
Let 
\begin{equation*}
N_I^k(\Delta)=
\begin{cases}
\wedge^i(F_I^{\perp}(\Delta)\otimes\C)\wedge(\wedge ^{k-i} N_{\C})\quad 
&\text{for}~ 0<i\leq k,\\
\wedge^k N_{\C}&\text{for}~ i=0,\\
0\quad &\text{for}~i>k.
\end{cases}
\end{equation*}
Then $N_I^k(\Delta)$ is a subspace of $\wedge^k N_{\C}$.
We have to notice that 
\begin{equation*}
N_I^k(\Delta)
\begin{cases}
\neq 0\quad\text{for}~~I\in S_k(\Delta)\\
=0\quad\text{for}~~I\notin S_k(\Delta)\\
\end{cases}
\end{equation*}
Moreover, we have
\begin{equation}
N_I^n(\Delta)=\wedge^n N_{\C}
\end{equation}
for all $I\in S_{\Delta}$.

\item
Let $$W_I^{k}(\Delta)=\rho(N_I^k(\Delta))$$ 
be a subspace of $H^0(X,\wedge^k\T_{X})$.
We define
\begin{equation*}
V_I^k(\Delta)=\chi^I\cdot W_I^k(\Delta)=\{\chi^I\cdot w\mid w\in W_I^k(\Delta)\}.
\end{equation*}
The elements in $V_I^k(\Delta)$ are considered as meromorphic $k$-vector fields on $X=X_{\Delta}$. In \cite{Hong 20}, it is proved that for any $I\in S_k(\Delta)$, the
$k$-vector fields in $V_I^k(\Delta)$ have moveable singularities. And we use them to represent the corresponding holomorphic $k$-vector fields. 
\end{enumerate}

The next theorem gives a description of the holomorphic polyvector fields on toric varieties. 
\begin{theorem}\cite{Hong 20}\label{polyvector-thm}
Let $X=X_ \Delta$ be a smooth compact toric variety associated with a fan $ \Delta$ in $N_\R$. Then we have the following decomposition
\begin{equation}\label{polyvect-thm-eqn1}
H^0(X,\wedge^k\T_{X})=\bigoplus_{I\in S_k(\Delta)}V_I^k(\Delta)
\end{equation}
for all $0\leq k\leq n$, where $V_I^k(\Delta)=V_{-I}^k$ for all $I\in S_k(\Delta)$.
\end{theorem}

\subsection{The Gerstenhaber algebra $(\A_{T}, \wedge, [,])$}
In the case of $X=T=(\C^*)^n$, the Gerstenhaber algebra structure of the polyvector fields (in algebraic version) has been studied in \cite{M-R 19}.  
Let us recall the results.

Any $k$-polyvector field (in algebraic version) on $X=T$ can be written as the finite sum
$$\sum_{I\in M}c_I\cdot\chi^I\cdot\rho(A_I),$$
where $c_I\in \C$, $A_I\in\wedge^k N_{\C}$.

Let 
\begin{equation*}
\C[M]=\{\sum_{I\in M}c_I\cdot\chi^I\mid c_I\in \C,~~~ \text{and there are only finite $c_I\neq 0$ in the sum}\},
\end{equation*} 
which can be considered as the group algebra of the abelian group $M$.
 
Let $$\A_{T}^k=\C[M]\otimes_{\C}\wedge^k N_{\C},$$  where $0\leq k\leq n$.
Then the vector space of $k$-polyvector field (in algebraic version) on $X=T$ is isomorphic to $\A_T^k$ by the map
$$\sum_{I\in M}c_I\cdot\chi^I\cdot\rho(A_I)\longleftrightarrow
\sum_{I\in M}c_I \cdot\chi^I\otimes A_I$$

Let $$\A_T=\bigoplus_{k\in\Z} \A^k_{T},$$ where $\A_T^k=0$ for $k<0$ or $k>n$.
The  Gerstenhaber algebra structure of the polyvector fields on $X=T$ can be described in the following way.

\begin{proposition}\cite{M-R 19}\label{MR-prop}
\begin{enumerate}
\item
The vector space $(\A_T,\cdot, [,])$ is a Gerstenhaber algebra, with the multiplication and the bracket given by the following way.
\begin{itemize}
\item
 The multiplication on $\A_T$ is given by 
$$(\chi^I\otimes A_I)\cdot (\chi^J\otimes A_J)= \chi^{I+J}\otimes (A_I\wedge A_J),$$
where $I,J\in M$, $A_I\in\wedge^k N_{\C}, A_J\in\wedge^l N_{\C}$.
\item The bracket $[,]$ on $\A_T$ is given by
$$[\chi^I\otimes A_I, \chi^J\otimes A_J]=(-1)^{k+1}\chi^{I+J}\otimes 
((\imath_{J} A_I)\wedge  A_J+(-1)^{k}A_I\wedge(\imath_I A_J)), $$
where $I,J\in M$, $A_I\in\wedge^k N_{\C}, A_J\in\wedge^l N_{\C}$, and 
$\imath_J(A_I)$ is the contraction of $J\in M\subset M_{\C}$ 
with $A_I\in\wedge^k N_{\C}$.
\end{itemize}
\item
Let $D: \A_T\rightarrow \A_T$ be a $\C$-linear operator given by 
$$D(\chi^I\otimes A_I)=-\chi^I\otimes (\imath_I A_I),$$
where $I\in M$ and $A_I\in\wedge^k N_{\C}$. Then $D$ is a BV operator for the Gerstenhaber algebra $(\A_T,\cdot, [,])$.
\end{enumerate}
\end{proposition}

\begin{remark}
In the second part of Proposition \ref{MR-prop}, the original statement in \cite{M-R 19} 
is that $-D$ is a BV operator for the Gerstenhaber algebra $(\A_T,\cdot, -[,])$, 
where the bracket on $\A_T$ is choosing to be
$-[,]$ (negative the bracket $[,]$), the operator $D$, the bracket $[,]$ and the multiplication $\cdot$ is shown in the first part of Proposition \ref{MR-prop}.
\end{remark}

\section{BV operators of the Gerstenhaber algebras of holomorphic polyvector fields on toric varieties}
\subsection{The Gerstenhaber algebra $(\A_{\Delta}, \wedge, [,])$}
 Let $X=X_{\Delta}$ be a smooth compact toric variety. 
Let $\A^k_{\Delta}=H^0(X,\wedge^k \T_X)$ be the vector space of holomorphic $k$-vector fields on $X$ for $k>0$. 
Let $\A^0_{\Delta}=H^0(X,\mathcal{O}_X)=\C$ and 
let $\A^k_{\Delta}=0$ for $k<0$ and $k>n$.
Let $$\A_{\Delta}=\bigoplus_{k\in\Z} \A^k_{\Delta}.$$
As shown in Example \ref{poly-Gerst}, $(\A_{\Delta}, \wedge, [,])$ is a Gerstenhaber algebra.

By Theorem \ref{polyvector-thm}, we have
\begin{equation}
\A^k_{\Delta}=\bigoplus_{I\in S_k(\Delta)}V_I^k(\Delta).
\end{equation}
The elements in $V_I^k(\Delta)$ can be written as 
$\chi^I\cdot\rho(A_I)$, where $I\in S_k(\Delta)$ and $A_I\in N_I^k(\Delta)$.

\begin{proposition}\label{Gerst-prop}
\begin{enumerate}
\item
The Gerstenhaber algebra structure on $(\A_{\Delta}, \wedge, [,])$ can be described by the following way. For any $\chi^I\cdot\rho(A_I)\in V_I^k(\Delta)$ and $\chi^J\cdot\rho(A_J)\in V_J^l(\Delta)$, we have
\begin{gather*}
(\chi^I\cdot\rho(A_I))\wedge (\chi^J\cdot\rho(A_J))=\chi^{I+J}\cdot\rho(A_I\wedge A_J),\\
[\chi^I\cdot\rho(A_I), \chi^J\cdot\rho(A_J)]=(-1)^{k+1}\chi^{I+J}\cdot
\rho((\imath_{J} A_I)\wedge  A_J+(-1)^{k}A_I\wedge(\imath_I A_J)),
\end{gather*}
where $I\in S_k(\Delta), A_I\in N_I^k(\Delta)$, and $J\in S_l(\Delta), A_J\in N_J^l(\Delta)$.
As a consequence, we have $V_I^k(\Delta)\wedge V_J^l(\Delta)\subseteq V_{I+J}^{k+l}(\Delta)$ and
$[V_I^k(\Delta), V_J^l(\Delta)]\subseteq V_{I+J}^{k+l-1}(\Delta)$.
\item 
Let $\gamma: \A_{\Delta}\rightarrow \A_T$ be a $\C$-linear map defined by 
$$\gamma(\chi^I\cdot \rho(A_I))=\chi^I\otimes A_I,$$ 
where $\chi^I\cdot\rho(A_I)\in V_I^k(\Delta)$, $I\in S_k(\Delta)$ 
and $A_I\in N_I^k(\Delta)$.
Then $\gamma: \A_{\Delta}\rightarrow \A_T$ is an injective map.
And $\gamma: (\A_{\Delta}, \wedge, [,])\rightarrow (\A_{T}, \wedge, [,])$ is a morphism of Gerstenhaber algebras. Thus $(\A_{\Delta}, \wedge, [,])$ can be considered as a 
sub-Gerstenhaber algebra of $(\A_{T}, \wedge, [,])$.
\end{enumerate}
\end{proposition}

Proposition \ref{Gerst-prop} can be proved in a similar way to Proposition \ref{MR-prop}, which can be find in \cite{M-R 19}. Here we skip the details.

\subsection{BV operators of the Gerstenhaber algebra $(\A_{\Delta}, \wedge, [,])$}
 Let $X=X_{\Delta}$ be a smooth compact toric variety. 
 A natural question is when the Gerstenhaber algebra
$(\A_{\Delta}, \wedge, [,])$ will be exact?
In other words, is there any degree $-1$ $\C$-linear operator 
$D:\A_{\Delta}\rightarrow \A_{\Delta}$ such that
\begin{enumerate}
\item
\begin{equation}\label{BV-eqn1}
[a,b]=(-1)^{|a|}(D(a\wedge b)-D(a)\wedge b-(-1)^{|a|}a\wedge D(b))
\end{equation}
for any $a,b\in \A_{\Delta}$,
\item 
\begin{equation}\label{BV-eqn2}
D^2=0.
\end{equation}
\end{enumerate}
Since $\A_{\Delta}^k=0$ for $k<0$ and $k>n$, we only need to consider the linear maps
$D: \A_{\Delta}^k\rightarrow\A_{\Delta}^{k-1}$ in the case of $1\leq k\leq n$.

\begin{lemma}\label{thm-lem1}
Let $D$ be a BV operator on the Gerstenhaber algebra $(\A_{\Delta}, \wedge, [,])$. 
Then there exists an element $\delta\in M_{\C}$ such that,
 for all $\chi^I\cdot\rho(A)\in V_I^k(\Delta)$ ($1\leq k\leq n$), we have
\begin{equation}
D(\chi^I\cdot\rho(A))=\chi^I\cdot\rho(\imath_{\delta-I}A),
\end{equation}
where $I\in S_k(\Delta)$ and $A\in N_I^k(\Delta)$.
\end{lemma}
\begin{proof}
\begin{itemize}
\item 
Let us consider the linear map $D: \A_{\Delta}^1\rightarrow \A_{\Delta}^0=\C$.
Let $\delta\in M_{\C}$ be defined by 
\begin{equation}\label{delta-eqn}
\langle\delta, x\rangle=D(\rho(x)),
\end{equation}
for all $x\in N_{\C}$, where $\rho(x)\in W\subseteq \A_{\Delta}^1$.
\item 
First, we prove the lemma in the case of $k=n$.
In Equation \eqref{BV-eqn1}, let $b=\chi^I\cdot\rho(A)$, where 
$I\in S_n(\Delta)=S_{\Delta}$ and $A\in N_I^n(\Delta)=\wedge^n N_{\C}$;
and let $a=\rho(x)$, where $x\in N_{\C}$.
By proposition \ref{Gerst-prop}, we have
\begin{gather*}
[a,b]=[\rho(x),\chi^I\cdot\rho(A)]=\langle I, x\rangle\chi^I\cdot \rho(A),
\end{gather*}
where $\langle I, x\rangle$ is the pairing of 
$I\in S_{\Delta}\subseteq M_{\R}\subseteq M_{\C}$ and $x\in N_{\C}$.
On the right hand of Equation \eqref{BV-eqn1}, 
we have $a\wedge b=0$ since that the complex dimension of $X_{\Delta}$ is $n$, 
and we have $D(a)=D(\rho(x))=\langle\delta, x\rangle$ by Equation \eqref{delta-eqn}.
By Equation \eqref{BV-eqn1}, we have
\begin{gather*}
\langle I, x\rangle\chi^I\cdot \rho(A)=\langle \delta, x\rangle)\chi^I\cdot\rho(A)-\rho(x)\wedge D(\chi^I\cdot\rho(A)),
\end{gather*}
which implies
\begin{equation}\label{D-eqn1}
\langle \delta-I, x\rangle\chi^I\cdot \rho(A)=\rho(x)\wedge D(\chi^I\cdot\rho(A)).
\end{equation}
On the other hand, we have $x\wedge A=0$, which implies
\begin{gather*} 
\imath_{\delta-I}(x\wedge A)=\langle\delta-I, x\rangle A- x\wedge(\imath_{\delta-I} A)=0.
\end{gather*}
Therefore we have $$\langle\delta-I, x\rangle A=x\wedge(\imath_{\delta-I} A).$$
As a consequence, we get
\begin{gather}\label{D-eqn2}
\langle \delta-I, x\rangle\chi^I\cdot \rho(A)=\chi^I\cdot\rho(\langle \delta-I, x\rangle A)
=\chi^I\cdot\rho(x\wedge(\imath_{\delta-I} A))=\chi^I\cdot\rho(x)\wedge\rho(\imath_{\delta-I} A).
\end{gather}
By Equation \eqref{D-eqn1} and Equation \eqref{D-eqn2}, we get
\begin{equation}\label{D-eqn3}
\rho(x)\wedge(D(\chi^I\cdot\rho(A))-\chi^I\cdot\rho(\imath_{\delta-I} A))=0.
\end{equation}
By Proposition \ref{Gerst-prop} and Equation \eqref{D-eqn3}, we have
\begin{equation}
x\wedge \gamma(D(\chi^I\cdot\rho(A))-\chi^I\cdot\rho(\imath_{\delta-I} A))=0
\end{equation}
for all $x\in N_{\C}$. Hence 
$$\gamma(D(\chi^I\cdot\rho(A))-\chi^I\cdot\rho(\imath_{\delta-I} A))=0.$$
As a consequence, we get 
$$D(\chi^I\cdot\rho(A))=\chi^I\cdot\rho(\imath_{\delta-I} A).$$
Therefore, we proved the lemma for $k=n$.

\item
Next we prove the lemma by recursion. Suppose the lemma holds in the case of 
$k=l $ $(l\geq2)$, we will prove it in the case of $k=l-1$.

In Equation \eqref{BV-eqn1}, let $b=\chi^I\cdot\rho(A)$, where 
$I\in S_{l-1}(\Delta)$ and $A\in N_I^{l-1}(\Delta)$; 
and let $a=\rho(x)$, where $x\in N_{\C}$.
By Proposition \ref{Gerst-prop}, we have
\begin{gather*}
[a,b]=[\rho(x),\chi^I\cdot\rho(A)]=\langle I, x\rangle\chi^I\cdot \rho(A).
\end{gather*}
On the right hand of Equation \eqref{BV-eqn1}, we have
\begin{gather*}
D(a\wedge b)=D(\rho(x)\wedge(\chi^I\cdot\rho(A)))=D(\chi^I\cdot\rho(x\wedge A))
=\chi^I\cdot\rho(\imath_{\delta-I}(x\wedge A)),\\
D(a)\wedge b=D(\rho(x))\chi^I\cdot\rho(A)=\langle\delta, x\rangle\chi^I\cdot\rho(A),\\
a\wedge D(b)=\rho(x)\wedge D(\chi^I\cdot \rho(A)).
\end{gather*}
By Equation \eqref{BV-eqn1}, we have
$$\langle I, x\rangle\chi^I\cdot \rho(A)=-\chi^I\cdot\rho(\imath_{\delta-I}(x\wedge A)) 
+\langle\delta, x\rangle\chi^I\cdot\rho(A)-\rho(x)\wedge D(\chi^I\cdot \rho(A)).$$
As a consequence, we have
\begin{gather*}
\rho(x)\wedge D(\chi^I\cdot\rho(A))=\chi^I\cdot(\langle\delta-I,x\rangle\rho(A)
-\rho(\imath_{\delta-I}(x\wedge A)))\\
=\chi^I\cdot\rho(x\wedge\imath_{\delta-I}A)
=\chi^I\cdot\rho(x)\wedge\rho(\imath_{\delta-I}A)=\rho(x)\wedge(\chi^I\cdot\rho(\imath_{\delta-I}A)).
\end{gather*}
Therefore, we have
$$\rho(x)\wedge(D(\chi^I\cdot \rho(A))-\chi^I\cdot\rho(\imath_{\delta-I}A))=0$$
for all $x\in N_{\C}$.
By similar reason as we have shown in the case of $k=n$, we get
$$D(\chi^I\cdot \rho(A))=\chi^I\cdot\rho(\imath_{\delta-I}A).$$
Thus if the lemma holds in the case of $k=l$, then it also holds in the case of $k=l-1$.
\end{itemize}
\end{proof}

To prove Lemma \ref{thm-lem2} and Lemma \ref{thm-lem3}, we need a basic fact in linear algebra, here we write it as a lemma.
\begin{lemma}[A basic fact of linear algebra]\label{liag-lem}
Let $V$ be a $n$-dimensional vector space over $\mathbf{k}$, 
and $F$ be a $i$-dimensional subspace of $V$. 
Let $A$ be a nonzero element in $\wedge^i F\cong\mathbf{k}$. 
Then for any $\alpha\in V^*$, we have that $\imath_{\alpha}A=0$ if and only if 
$\alpha\in F^{\perp}$, where $F^{\perp}\subseteq V^*$ is the annihilator of $F$. 
\end{lemma}

\begin{lemma}\label{thm-lem2}
Let $\delta$ be an element in $M_{\C}$. Suppose that for
all $\chi^I\cdot\rho(A)\in V_I^k(\Delta)$ $(1\leq k\leq n)$, 
we have $\chi^I\cdot\rho(\imath_{\delta-I}A)\in V_I^{k-1}(\Delta)$,
where $I\in S_k(\Delta)$ and $A\in N_I^k(\Delta)$.
\begin{enumerate}
 \item
 We have
\begin{equation}\label{thm-lem2-eqn1}
\delta\in\bigcap_{I\in S_{\Delta}}\hat{F}_I(\Delta),
\end{equation}
\item
If $I\in M$ is a vertex of $P_{\Delta}$, then we have $\delta=I$.
Thus the polytope $P_{\Delta}\subset M_{\R}$ has at most one vertex belonging to $M$.
\end{enumerate}
\end{lemma}
\begin{proof}
\begin{enumerate}
\item
We will prove that for any $I\in S_{\Delta}$, we have $\delta\in \hat{F}_I(\Delta)$.

Suppose that $I\in S(\Delta, i)$. Then we have $\dim F_I(\Delta)=n-i$ 
 and $\dim F_I^{\perp}(\Delta)=i$. 
According to the assumption, for any $\chi^I\cdot\rho(A)\in V_I^k(\Delta)$, 
we have $\chi^I\cdot\rho(\imath_{\delta-I}A)\in V_I^{k-1}(\Delta)$.
Hence for any $A\in N_I^k(\Delta)$, we have $\imath_{\delta-I}A\in N_I^{k-1}(\Delta)$.
In the case of $k=i$, we have $N_I^{i-1}(\Delta)=0$.
Let $A$ be a nonzero element in 
$N_I^i(\Delta)=\wedge^i (F_I^{\perp}(\Delta)\otimes\C)\cong\C$.
Then we have
\begin{equation}\label{thm-lem2-eqn2}
\imath_{\delta-I}A=0.
\end{equation}
By Lemma \ref{liag-lem} and Equation \eqref{thm-lem2-eqn2}, we have
\begin{equation}\label{thm-lem2-eqn3}
\delta-I\in (F_I^{\perp}(\Delta)\otimes\C)^{\perp}.
\end{equation}
As a consequence, we get $\delta\in \hat{F}_I(\Delta)$ for all $I\in S(\Delta,i)$.

Since $S_\Delta=\bigcup_{0\leq i\leq n} S(\Delta, i)$, we get
$$\delta\in\bigcap_{I\in S_{\Delta}}\hat{F}_I(\Delta).$$ 
\item
Suppose that  $I\in M$ is a vertex of the polytope $P_{\Delta}$.
Then we have $\delta\in S_{\Delta}$ and $\hat{F}_I(\Delta)=\{I\}$.
Since 
$$\delta\in\bigcap_{I\in S_{\Delta}}\hat{F}_I(\Delta),$$
we have $\delta=I$.
Thus the polytope $P_{\Delta}\subset M_{\C}$ has at most one vertex belonging to $M$.
\end{enumerate}
\end{proof}

\begin{lemma}\label{thm-lem3}
Let $\delta\in\bigcap_{I\in S_{\Delta}}\hat{F}_I(\Delta)$. 
Then for all $\chi^I\cdot\rho(A)\in V_I^k(\Delta)$ $(1\leq k\leq n)$, 
we have $\chi^I\cdot\rho(\imath_{\delta-I}A)\in V_I^{k-1}(\Delta)$,
where $I\in S_k(\Delta)$ and $A\in N_I^k(\Delta)$.
\end{lemma}
\begin{proof}
To prove the lemma, it is enough for us to prove that $\imath_{\delta-I}A\in N_I^{k-1}(\Delta)$ for all $I\in S_k(\Delta)$ and $A\in N_I^k(\Delta)$. 
Suppose that $I\in S(\Delta, i)$. Then we have $k\leq i$ since $I\in S_k(\Delta)$.

\begin{itemize}
\item
We first prove the lemma in the case of $k=i$. In the case of $k=i$, we have 
\begin{gather*}
N_I^i(\Delta)=
\wedge^i (F_I^{\perp}(\Delta)\otimes\C)\cong\C \quad\text{and}\quad N_I^{i-1}(\Delta)=0 .
\end{gather*}
Suppose that $A_1$ is a nonzero element in $N_I^i(\Delta)=
\wedge^i (F_I^{\perp}(\Delta)\otimes\C)\cong\C$.
Since $\delta\in \hat{F}_I(\Delta)$, we have $\delta-I\in (F_I^{\perp}(\Delta)\otimes\C)^{\perp}$.
By Lemma \ref{liag-lem}, we get $$\imath_{\delta-I}A_1=0.$$
Since $N_I^i(\Delta)=\C A_1$, we have proved the lemma in the case of $k=i$.
\item
In the case of $k>i$, we have 
$$N_I^k(\Delta)=N_I^i(\Delta)\wedge(\wedge^{k-i}N_{\C}).$$
Any element  $A\in N_I^k(\Delta)$ can be written as 
$$A=A_1\wedge A_2,$$
where $A_2\in \wedge^{k-i} N_{\C}$.
We have
\begin{equation}
\imath_{\delta-I}A=
(\imath_{\delta-I}A_1)\wedge A_2+(-1)^i A_1\wedge(\imath_{\delta-I}A_2)
=(-1)^i A_1\wedge(\imath_{\delta-I}A_2).
\end{equation}
Since $A_1\in N_I^i(\Delta)$ and $\imath_{\delta-I}A_2\in\wedge^{k-i-1}N_{\C}$,
we have $$\imath_{\delta-I}A\in N_I^{k-1}(\Delta).$$
\end{itemize}
\end{proof}

\begin{lemma}\label{thm-lem4}
Let $\delta$ be an element in $M_{\C}$. Let $D$ be a linear operator defined on the Gerstenhaber algebra 
$(\A_{\Delta}, \wedge, [,])$ satisfying
\begin{equation*}
D(\chi^I\cdot\rho(A))=\chi^I\cdot\rho(\imath_{\delta-I}A)
\end{equation*}
for all $\chi^I\cdot\rho(A)\in V_I^k(\Delta)$ $(1\leq k\leq n)$, 
where $I\in S_k(\Delta)$ and $A\in N_I^k(\Delta)$.
Then $D$ satisfies Equation \eqref{BV-eqn1} and \eqref{BV-eqn2}.
\end{lemma}
\begin{proof}
\begin{itemize}
\item
Since that $\A_{\Delta}^k=0$ for $k<0$ and $k>n$, we only need to verify that 
$D^2:\A_{\Delta}^k\rightarrow\A_{\Delta}^{k-2}$ equal to zero for $2\leq k\leq n$.

As $D$ is a linear map from the vector space
 $\A_{\Delta}^k=\bigoplus_{I\in S_k(\Delta)}V_I^k(\Delta)$ to
the vector space $\A_{\Delta}^{k-1}=\bigoplus_{I\in S_{k-1}(\Delta)}V_I^{k-1}(\Delta)$,
we get that
\begin{gather*}
D(\chi^I\cdot\rho(A))=\chi^I\cdot\rho(\imath_{\delta-I}A)\in V_I^{k-1}(\Delta)
\end{gather*}
for all $\chi^I\cdot\rho(A)\in V_I^k(\Delta)$ $(1\leq k\leq n)$.
For any $\chi^I\cdot\rho(A)\in V_I^k(\Delta)$, suppose that $I\in S(\Delta, i)$, 
then we have $i\leq k$ since $I\in S_k(\Delta)$.

In the case of $k<i+2$, since $D^2(\chi^I\cdot\rho(A))\in V_I^{k-2}(\Delta)$ and $V_I^{k-2}(\Delta)=0$, we get that $D^2(\chi^I\cdot\rho(A))=0$. 
And in the case of $k\geq i+2$, we have
\begin{gather*}
D^2(\chi^I\cdot\rho(A))=D(\chi^I\cdot\rho(\imath_{\delta-I}A))
=\chi^I\cdot\rho(\imath_{\delta-I}(\imath_{\delta-I}A))=0.
\end{gather*}
Since  $\A_{\Delta}^k=\bigoplus_{I\in S_k(\Delta)}V_I^k(\Delta)$,
we get $D^2=0$.

\item
If $a\in \A^0_{\Delta}=\C$ or $b\in \A^0_{\Delta}=\C$, Equation \eqref{BV-eqn2} holds obviously.

Next we will prove that Equation  \eqref{BV-eqn2} holds for any $a\in V_I^k(\Delta)$ 
and $b\in V_J^l(\Delta)$, where $1\leq k, l\leq n$, $I\in S_k(\Delta)$ and 
$J\in S_l(\Delta)$.

Suppose that $a=\chi^I\cdot\rho(A_I)$ and $b=\chi^J\cdot\rho(A_J)$, where $A_I\in N_I^k(\Delta)$ and $A_J\in N_J^l(\Delta)$.
By Proposition \ref{Gerst-prop}, we have
\begin{gather}\label{BVeqn-lem-eqn1}
[a,b]=[\chi^I\cdot\rho(A_I), \chi^J\cdot\rho(A_J)]=(-1)^{k+1}\chi^{I+J}\cdot
\rho((\imath_{J} A_I)\wedge  A_J+(-1)^{k}A_I\wedge(\imath_I A_J)).
\end{gather}

On the other hand, we have
\begin{equation}\label{BVeqn-lem-eqn2}
D(a\wedge b)=D((\chi^I\cdot\rho(A_I))\wedge(\chi^J\cdot\rho(A_J)))=D(\chi^{I+J}\cdot\rho(A_I\wedge A_J))=\chi^{I+J}\cdot\rho(\imath_{\delta-I-J(}A_I\wedge A_J)).
\end{equation}
We notice that Equation \eqref{BVeqn-lem-eqn2} holds even in the case of $k+l>n$ since that $a\wedge b=0$ and $A_I\wedge A_j=0$ if $k+l>n$.
Moreover, we have
\begin{equation}\label{BVeqn-lem-eqn3}
D(a)\wedge b=D(\chi^I\cdot\rho(A_I))\wedge(\chi^J\cdot\rho(A_J))
=(\chi^I\cdot\rho(\imath_{\delta-I}A_I))\wedge(\chi^J\cdot\rho(A_J))=
\chi^{I+J}\cdot\rho((\imath_{\delta-I}A_I)\wedge A_J)
\end{equation}
and
\begin{equation}\label{BVeqn-lem-eqn4}
a\wedge D(b)=\chi^I\cdot\rho(A_I)\wedge D(\chi^J\cdot\rho(A_J))
=(\chi^I\cdot\rho(A_I))\wedge(\chi^J\cdot\rho(\imath_{\delta-J}A_J))
=\chi^{I+J}\cdot\rho(A_I\wedge(\imath_{\delta-J}A_J)).
\end{equation}

By Equation \eqref{BVeqn-lem-eqn2}, Equation \eqref{BVeqn-lem-eqn3} and Equation \eqref{BVeqn-lem-eqn4}, we get that
\begin{gather*}
D(a\wedge b)-D(a)\wedge b-(-1)^k a\wedge D(b)\\
=\chi^{I+J}\cdot(\rho(\imath_{\delta-I-J(}A_I\wedge A_J))-\rho((\imath_{\delta-I}A_I)\wedge A_J)-(-1)^k\rho(A_I\wedge(\imath_{\delta-J}A_J)))\\
=\chi^{I+J}\cdot\rho(\imath_{\delta-I-J}(A_I\wedge A_J)-(\imath_{\delta-I}A_I)\wedge A_J-(-1)^k A_I\wedge(\imath_{\delta-J}A_J)).
\end{gather*}
Since that
\begin{gather*}
\imath_{\delta}(A_I\wedge A_J)-\imath_{\delta} A_I\wedge A_J-(-1)^k A_I\wedge \imath_{\delta} A_J=0,\\
\imath_{-I}(A_I\wedge A_J)-(\imath_{-I} A_I)\wedge A_J
=(-1)^k A_I\wedge(\imath_{-I}A_J),\\
\imath_{-J}(A_I\wedge A_J)-(-1)^k A_I\wedge(\imath_{-J}A_J)
=\imath_{-J}A_I\wedge A_J,
\end{gather*}
we get
\begin{equation}\label{BVeqn-lem-eqn5}
D(a\wedge b)-D(a)\wedge b-(-1)^k a\wedge D(b)
=\chi^{I+J}\cdot\rho((-1)^k A_I\wedge(\imath_{-I}A_J)+\imath_{-J}A_I\wedge A_J).
\end{equation}
By Equation \eqref{BVeqn-lem-eqn2} and Equation \eqref{BVeqn-lem-eqn5}, we get
\begin{equation*}
[a,b]=(-1)^k(D(a\wedge b)-D(a)\wedge b-(-1)^k a\wedge D(b)).
\end{equation*}
\end{itemize}
\end{proof}

\begin{theorem}
Let $X_{\Delta}$ be a smooth compact toric variety.  
Let $(\A_{\Delta}, \wedge, [,])$ be the Gerstenhaber algebra associated to $X_{\Delta}$.
\begin{enumerate}
\item
The necessary and sufficient condition for the existence of the BV operator is that
\begin{equation}\label{thm-eqn1}
\bigcap_{I\in S_{\Delta}}\hat{F}_I(\Delta)\neq\emptyset.
\end{equation}
\item
If Equation \eqref{thm-eqn1} holds, 
choosing any element $\delta\in\bigcap_{I\in S_{\Delta}}\hat{F}_I(\Delta)$, 
we can define a BV-operator $D: \A_{\Delta}\rightarrow\A_{\Delta}$ by
\begin{equation}\label{thm-eqn2}
D(\chi^I\cdot\rho(A))=\chi^I\cdot\rho(\imath_{\delta-I}A),
\end{equation}
where $\chi^I\cdot\rho(A)\in V_I^k(\Delta)$ $(1\leq k\leq n)$, $I\in S_k(\Delta)$ and $A\in N_I^k(\Delta)$.
\item
All the BV operators of the Gerstenhaber algebra $(\A_{\Delta}, \wedge, [,])$ 
are in the form as  shown in Equation \eqref{thm-eqn2}
\end{enumerate}
\end{theorem}
\begin{proof}
\begin{enumerate}
\item
\begin{itemize}
\item The necessary condition:\\
Suppose that $D: A_{\Delta}\rightarrow A_{\Delta}$ is a BV-operator of 
Gerstenhaber algebra $(\A_{\Delta}, \wedge, [,])$. 
By Lemma \ref{thm-lem1}, there exist an element $\delta\in M_{\C}$ such that,
 for all $\chi^I\cdot\rho(A)\in V_I^k(\Delta)$ ($1\leq k\leq n$), we have
\begin{equation*}
D(\chi^I\cdot\rho(A))=\chi^I\cdot\rho(\imath_{\delta-I}A),
\end{equation*}
where $I\in S_k(\Delta)$ and $A\in N_I^k(\Delta)$.
By Theorem \ref{polyvector-thm}, to make $D: A_{\Delta}\rightarrow A_{\Delta}$ to be a well defined operator, we have
\begin{equation*}
D(\chi^I\cdot\rho(A))=\chi^I\cdot\rho(\imath_{\delta-I}A)\in V_I^{k-1}(\Delta)
\end{equation*}
for all $I\in S_k(\Delta)$ and $A\in N_I^k(\Delta)$.
And by Lemma \ref{thm-lem3}, we get
\begin{equation*}
\delta\in \bigcap_{I\in S_{\Delta}}\hat{F}_I(\Delta).
\end{equation*}
Hence we have
$$\bigcap_{I\in S_{\Delta}}\hat{F}_I(\Delta)\neq\emptyset.$$
\item The sufficient condition:\\
If Equation \eqref{thm-eqn1} holds, 
choosing any element $\delta\in\bigcap_{I\in S_{\Delta}}\hat{F}_I(\Delta)$, 
we can define a $\C$-linear operator $D: \A_{\Delta}\rightarrow\A_{\Delta}$ by
\begin{equation*}
D(\chi^I\cdot\rho(A))=\chi^I\cdot\rho(\imath_{\delta-I}A),
\end{equation*}
where $\chi^I\cdot\rho(A)\in V_I^k(\Delta)$ $(1\leq k\leq n)$, 
$I\in S_k(\Delta)$ and $A\in N_I^k(\Delta)$. By Lemma \ref{thm-lem3}, 
$D: A_{\Delta}\rightarrow A_{\Delta}$ is a well defined operator since 
$D(V_I^k(\Delta))\subseteq V_I^{k-1}(\Delta)$.
By Lemma \ref{thm-lem4}, $D$ is a BV-operator of the Gerstenhaber algebra 
$(\A_{\Delta}, \wedge, [,])$.
\end{itemize}
\item It is already proved above in the sufficient condition part.
\item It follows directly from Lemma \ref{thm-lem1}.
\end{enumerate}
\end{proof}

\end{document}